\documentclass[12pt]{amsart}
\usepackage{amsmath,amssymb,amsbsy,amsfonts,latexsym,amsopn,amstext,
                                               amsxtra,euscript,amscd,bm}
                    
\usepackage{url}
\usepackage[colorlinks,linkcolor=blue,anchorcolor=blue,citecolor=blue]{hyperref}
\usepackage{color}

\begin{document}
\newtheorem{problem}{Problem}
\newtheorem{theorem}{Theorem}
\newtheorem{lemma}[theorem]{Lemma}
\newtheorem{claim}[theorem]{Claim}
\newtheorem{cor}[theorem]{Corollary}
\newtheorem{prop}[theorem]{Proposition}
\newtheorem{definition}{Definition}
\newtheorem{question}[theorem]{Question}
\newtheorem{conj}{Conjecture}
\newtheorem{hypothesis}{Hypothesis}
\def\cA{{\mathcal A}}
\def\cB{{\mathcal B}}
\def\cC{{\mathcal C}}
\def\cD{{\mathcal D}}
\def\cE{{\mathcal E}}
\def\cF{{\mathcal F}}
\def\cG{{\mathcal G}}
\def\cH{{\mathcal H}}
\def\cI{{\mathcal I}}
\def\cJ{{\mathcal J}}
\def\cK{{\mathcal K}}
\def\cL{{\mathcal L}}
\def\cM{{\mathcal M}}
\def\cN{{\mathcal N}}
\def\cO{{\mathcal O}}
\def\cP{{\mathcal P}}
\def\cQ{{\mathcal Q}}
\def\cR{{\mathcal R}}
\def\cS{{\mathcal S}}
\def\cT{{\mathcal T}}
\def\cU{{\mathcal U}}
\def\cV{{\mathcal V}}
\def\cW{{\mathcal W}}
\def\cX{{\mathcal X}}
\def\cY{{\mathcal Y}}
\def\cZ{{\mathcal Z}}

\def\A{{\mathbb A}}
\def\B{{\mathbb B}}
\def\C{{\mathbb C}}
\def\D{{\mathbb D}}
\def\E{{\mathbb E}}
\def\F{{\mathbb F}}
\def\G{{\mathbb G}}
\def\I{{\mathbb I}}
\def\J{{\mathbb J}}
\def\K{{\mathbb K}}
\def\L{{\mathbb L}}
\def\M{{\mathbb M}}
\def\N{{\mathbb N}}
\def\O{{\mathbb O}}
\def\P{{\mathbb P}}
\def\Q{{\mathbb Q}}
\def\R{{\mathbb R}}
\def\S{{\mathbb S}}
\def\T{{\mathbb T}}
\def\U{{\mathbb U}}
\def\V{{\mathbb V}}
\def\W{{\mathbb W}}
\def\X{{\mathbb X}}
\def\Y{{\mathbb Y}}
\def\Z{{\mathbb Z}}

\def\ep{{\mathbf{e}}_p}
\def\em{{\mathbf{e}}_m}
\def\eq{{\mathbf{e}}_q}

\def\scr{\scriptstyle}
\def\\{\cr}
\def\({\left(}
\def\){\right)}
\def\[{\left[}
\def\]{\right]}
\def\<{\langle}
\def\>{\rangle}
\def\fl#1{\left\lfloor#1\right\rfloor}
\def\rf#1{\left\lceil#1\right\rceil}
\def\le{\leqslant}
\def\ge{\geqslant}
\def\eps{\varepsilon}
\def\mand{\qquad\mbox{and}\qquad}

\def\sssum{\mathop{\sum\ \sum\ \sum}}
\def\ssum{\mathop{\sum\, \sum}}
\def\ssumw{\mathop{\sum\qquad \sum}}

\def\vec#1{\mathbf{#1}}
\def\inv#1{\overline{#1}}
\def\num#1{\mathrm{num}(#1)}
\def\dist{\mathrm{dist}}

\def\fA{{\mathfrak A}}
\def\fB{{\mathfrak B}}
\def\fC{{\mathfrak C}}
\def\fU{{\mathfrak U}}
\def\fV{{\mathfrak V}}

\newcommand{\bflambda}{{\boldsymbol{\lambda}}}
\newcommand{\bfxi}{{\boldsymbol{\xi}}}
\newcommand{\bfrho}{{\boldsymbol{\rho}}}
\newcommand{\bfnu}{{\boldsymbol{\nu}}}

\def\GL{\mathrm{GL}}
\def\SL{\mathrm{SL}}

\def\Hba{\overline{\cH}_{a,m}}
\def\Hta{\widetilde{\cH}_{a,m}}
\def\Hb1{\overline{\cH}_{m}}
\def\Ht1{\widetilde{\cH}_{m}}

\def\flp#1{{\left\langle#1\right\rangle}_p}
\def\flm#1{{\left\langle#1\right\rangle}_m}
\def\dmod#1#2{\left\|#1\right\|_{#2}}
\def\dmodq#1{\left\|#1\right\|_q}

\def\Zm{\Z/m\Z}

\def\Err{{\mathbf{E}}}

\newcommand{\comm}[1]{\marginpar{%
\vskip-\baselineskip 
\raggedright\footnotesize
\itshape\hrule\smallskip#1\par\smallskip\hrule}}

\def\xxx{\vskip5pt\hrule\vskip5pt}


\title{Large values of the error term in the prime number theorem}

\author[B. Kerr] {Bryce Kerr}
\address{BK: School of Science, University of New South Wales, Canberra, ACT, 2612, Australia}
\email{bryce.kerr@unsw.edu.au}

\begin{abstract}
Assume the Riemann hypothesis throughout. We obtain some new estimates for the size of the set of large values of the error term in the prime number theorem. Our argument is based on an analysis of the behavior of zeros of the Riemann zeta function in Bohr sets.
\end{abstract}


\maketitle
\section{Introduction}

We assume the Riemann hypothesis throughout. In particular, each complex zero $\rho$ of the Rieman zeta function $\zeta$ may be represented in the form
\begin{align}
\label{eq:gamma-def}
\rho=\frac{1}{2}+i\gamma, \quad \gamma\in \R.
\end{align}
A classic result of von Koch~\cite{vonK} states that
\begin{align}
\label{eq:psi}
\psi(x)=x+O(x^{1/2}(\log{x})^2),
\end{align}
where
$$\psi(x)=\sum_{n\le x}\Lambda(n)$$
denotes the Chebyshev function and $\Lambda(n)$ the von Mangoldt function.  The sharpest explicit form of~\eqref{eq:psi} is due to Schoenfeld~\cite{Sch}, who showed for large enough $x$
\begin{align}
\label{eq:scho}
|\psi(x)-x|\le \frac{1}{8\pi}x^{1/2}(\log{x})^{2}.
\end{align}
It is expected that the estimate~\eqref{eq:scho} does not represent the true rate of growth of $\psi(x)-x$. Monach and Montgomery, see~\cite[Chapter~15]{MV}, have shown that a strong form of the linear independence conjecture implies 
\begin{align}
\label{eq:suplim}
\limsup_{x\rightarrow \infty}\frac{\psi(x)-x}{x^{1/2}(\log\log\log{x})^{3}}\ge \frac{1}{2\pi},
\end{align}
and
\begin{align}
\label{eq:inflim}
\liminf_{x\rightarrow \infty}\frac{\psi(x)-x}{x^{1/2}(\log\log\log{x})^{3}}\le -\frac{1}{2\pi}.
\end{align}
There has been some speculation that both~\eqref{eq:suplim} and~\eqref{eq:inflim} may be strengthened to equality, see~\cite[pg. 484]{MV}. 

A fundamental problem is to improve the constant $1/8\pi$ in~\eqref{eq:scho}. With $\gamma$ is as in~\eqref{eq:gamma-def}, this is more or less equivalent to bounding exponential sums of the form
\begin{align}
\label{eq:abc-47}
\sum_{0<\gamma \le T}x^{i\gamma}
\end{align}
in a range of parameters 
\begin{align}
\label{eq:range of parameters}
T^{1/2}\le x\le T^{A}
\end{align}
with $A$ as large as possible. Our main source of knowledge about the sum~\eqref{eq:abc-47} comes from the prime numbers via contour integration. This is known as the Landau-Gonek formula~\cite{Gonek} and implies the following estimate 
$$\sum_{0\gamma \le T}x^{i\gamma}\ll \left(\frac{T}{x^{1/2}}+x^{1/2}\right)x^{o(1)}.$$
The last term on the right comes from the error in truncating an integral involving $x^{s}$ along the line $\Re{(s)}=1$, leaving no clear path to establishing a suitable estimate for~\eqref{eq:abc-47} in the range of parameters~\eqref{eq:range of parameters} via classical techniques.

 The only progress on this problem has been conditional on Montgomery's pair correlation conjecture, which states that for any fixed $0<\alpha<\beta$ we have
\begin{align}
\label{eq:pcc}
\nonumber &\left|\left\{ |\gamma|,|\gamma'|\le T \ : \ \gamma-\gamma'\in \left[\frac{2\pi \alpha}{\log{T}},\frac{2\pi \beta}{\log{T}}\right] \right\}\right| \\ & \quad \quad =(1+o(1))\int_{\alpha}^{\beta}\left(1-\left(\frac{\sin{\pi u}}{u}\right)^2 \right)du\left(\frac{T\log{T}}{\pi}\right).
\end{align}
 Gallagher and Mueller~\cite{GM} have show that~\eqref{eq:pcc} implies 
\begin{align}
\label{eq:psi-1}
\psi(x)=x+o(x^{1/2}(\log{x})^2)
\end{align}
and developing a precise relationship between the error terms in~\eqref{eq:pcc} and~\eqref{eq:psi-1} has been the subject of a number of works, see~\cite{GM,HBG,GS1,GS,HB,LL,LPZ}.  We refer the reader to Odlyzko~\cite{Odl} for numerical verifications of~\eqref{eq:pcc} and Hejhal~\cite{Hej} and Rudnick and Sarnak~\cite{RS} for investigations into higher level correlations between $\gamma$'s. There has been very little progress towards establishing~\eqref{eq:pcc} and this motivates the problem of estimating the size of the set of exceptions to~\eqref{eq:psi-1}. For example Gallagher~\cite{Gal} has shown

\begin{align}
\label{eq:psi-11}
\psi(x)=x+O(x^{1/2}(\log\log{x})^2),
\end{align}
except for a set of finite logarithmic measure. 

The first distributional estimates for $\psi(x)-x$ were obtained by Wintner~\cite{Wint}, who showed the existence of a measure $\nu$ such that for all absolutely continuous functions $f$
\begin{align}
\label{eq:123}
\lim_{y\rightarrow \infty}\frac{1}{y}\int_{0}^{y}f\left(\frac{\psi(e^{u})-e^{u}}{e^{u/2}}\right)du=\int_{\R}f(x)d\nu(x).
\end{align}
 We refer the reader to~\cite{ANS,Ng,RS} for various extensions of Wintner's result. It is difficult to establish properties of the measure $\nu$ in~\eqref{eq:123} without information about the diophantine nature of $\gamma$'s. However, it is possible to estimate the rate of decay of $\nu$.

Calculations of Wintner~\cite{Wint} imply the existence of an absolute constant $c$ such that for any even integer $k$ \begin{align}
\label{eq:wint}
\int_{0}^{X}\left|\psi(x)-x\right|^{k}dx\ll (c k^2)^{k}X^{k+1}.
\end{align}
Wintner did not give an explicit value of $c$. Evaluating the sum given by~\cite[Equation (11)]{Wint} shows one may take
\begin{align}
\label{eq:c-const}
c=(1+o(1))\frac{2}{\pi},
\end{align}
where the term $o(1)\rightarrow 0$ as $k\rightarrow \infty$.

We refer the reader to~\cite{RF} for progress on a related problem of estimating moments from below.

One consequence of~\eqref{eq:wint} is the following large values estimate:
\begin{prop}
\label{prop:largevalue}
Let $c$ be as in~\eqref{eq:wint} and $\mu$ denote the Lebesgue measure. For any $\varepsilon>0$ we have 
\begin{align*}
\mu\left(\left\{x \le X \ : \ |\psi(x)-x|\ge \varepsilon x^{1/2}(\log{x})^{2}\right\}\right)\ll X^{1-c'\varepsilon^{1/2}},
\end{align*}
where 
$$c'=2\exp(-c/2-1)+o(1),$$ and
 $o(1)\rightarrow 0$ as $X\rightarrow 0$.
\end{prop}
In particular, with $c$ as in~\eqref{eq:c-const}, we may take 
$$c'\approx 0.53517...$$

In this paper we investigate the extent to which it is possible to improve on Proposition~\ref{prop:largevalue}. 
By considering the behaviour of $\gamma$'s in Bohr sets, we show that such an improvement is possible. A second consequence of our work is that the set of large values $x\in [X,2X]$ of $|\psi(x)-x|$ must concentrate into a small number of short intervals $\cI$, see Theorem~\ref{thm:main1}.

\subsection{Outline of our argument}
In order to show that large values of $|\psi(x)-x|$ concentrate into short intervals, we analyse the behaviour of $\gamma$'s in Bohr sets. Using the analytic approximation

$$\psi(x)-x\approx x^{1/2}\sum_{|\gamma|\le X^{1/2}}\frac{x^{i\gamma}}{1/2+i\gamma},$$
we see that each value of $x$ satisfying 
\begin{align}
\label{eq:psi-c}
|\psi(x)-x|\ge \varepsilon x^{1/2}(\log{x})^{2},
\end{align}
 corresponds to many values of $t\le X^{1/2}$ satisfying
\begin{align}
\label{eq:qwe}
\left|\sum_{0\le \gamma \le t}x^{i\gamma}\right|\ge \frac{\varepsilon }{2}N(t),
\end{align}
where 
\begin{align*}
N(t)=|\{ 0\le \gamma \le t \ : \ \zeta(1/2+i\gamma)=0\}|.
\end{align*}
An application of the pigeonhole principle allows us to obtain some $t$ such that for most $x$ satisfying~\eqref{eq:psi-c} we have~\eqref{eq:qwe}.

 A well known principle in combinatorics asserts that large exponential sums concentrate into Bohr sets. Each $x$ satisfying~\eqref{eq:qwe} corresponds to a pair $(x,\beta_x)$ satisfying 
\begin{align}
\label{eq:Bohr-1}
\left|\left\{ 0\le \gamma \le t \ : \left\|\frac{\log{x}}{2\pi} \gamma+\beta_x\right\|\le \delta \right\}\right|\ge 2\delta\left(1+\frac{\varepsilon}{8}\right)N(t),
\end{align}
where $\|.\|$ denotes distance to the nearest integer, see Lemma~\ref{lem:concentration}. The inequality~\eqref{eq:Bohr-1} is larger than expected by a factor $(1+\varepsilon/8)$.  We refer the reader to work of  Ford and  Zaharescu~\cite{FZ} and Ford, Soundararajan and Zaharescu~\cite{FSZ} for various results and conjectures which suggest that trying to obtain a contradiction directly from~\eqref{eq:Bohr-1} would be very difficult. Instead, we proceed by assuming many pairs $(x,\beta_x)$ satisfy~\eqref{eq:Bohr-1}. This allows us to amplify the factor $(1+\varepsilon/8)$.

H\"{o}lder's inequality implies that for any integer $k$, there exists many $k$-tuples
\begin{align}
\label{eq:xb}
x_1,\dots,x_k,\beta_1,\dots,\beta_k
\end{align}
 such that 
\begin{align}
\label{eq:Bohr-2}
\left|\left\{ 0\le \gamma \le t \ : \left\|\frac{\log{x_j}}{2\pi} \gamma+\beta_j\right\|\le \delta, \ \ 1\le j \le k \right\}\right|\ge (2\delta)^{k}\left(1+\frac{\varepsilon}{8}\right)^{k}N(t).
\end{align}

We dext discuss some heuristics regarding Bohr sets and refer the reader to Lemma~\ref{lem:bohr-sequence} for a precise statement of the argument sketched below. 

For most choices of tuples $(x_1,\dots,x_k)$, we have 
\begin{align}
\label{eq:aabc-99}
\mu\left(\left\{ y\in [0,t] \ : \left\|\frac{\log{x_j}}{2\pi} y+\beta_j\right\|\le \delta, \ \ 1\le j \le k \right\}\right)\approx (2\delta)^{k}t,
\end{align}
with some uniformity in the parameter $\delta$. This can be seen by interpreting the above  volume calculation~\eqref{eq:aabc-99} in terms of integer points close to the line 
$$\left(\frac{\log{x_1}}{2\pi} y+\beta_1,\dots,\frac{\log{x_k}}{2\pi} y+\beta_k\right) \quad 0\le y \le t,$$
and applying a transference theorem from the geometry of numbers. 

 Consider the set~\eqref{eq:aabc-99} as a union of $N$ intervals $I_1,\dots,I_N$ 
\begin{align}
\label{eq:intun}
\left\{ y\in [0,t] \ : \left\|\frac{\log{x_j}}{2\pi} y+\beta_j\right\|\le \rho, \ \ 1\le j \le k \right\}=\bigcup_{j=1}^{N}I_j.
\end{align}
If we extend the endpoints of each $I_j$ by a factor $\approx \eta/\log{X}$ then each point $y$ in the resulting set satisfies 
$$\left\|\frac{\log{x_j}}{2\pi} y+\beta_j\right\|\le \delta+\eta, \quad 1\le j \le k.$$
Hence~\eqref{eq:aabc-99} implies roughtly that
\begin{align*}
(2\delta)^{k}t+\frac{2\eta}{\log{X}}N\approx 2^{k}(\delta+\eta)^{k}t.
\end{align*}
In particular~\eqref{eq:intun} is the union of 
$N\approx (2\delta)^{k-1}(\log{X})t,$ intervals of length 
\begin{align}
\label{eq:short-length}
\approx \frac{\delta}{\log{X}}.
\end{align}
From~\eqref{eq:Bohr-2} there are many intervals  of length~\eqref{eq:short-length} containing $\approx (1+\varepsilon)^{k}$ zeros of $\zeta$. Montgomery's work towards the pair correlation conjecture implies the average density of zeros in intervals of length~\eqref{eq:short-length} is $O(1)$. This allows us to obtain a contradiction by taking $k$ sufficiently large.

 \section{Main results}
\begin{theorem}
\label{thm:main1}
Let $\varepsilon,\delta>0$ be small and $X$ sufficiently large. Suppose $\cX\subseteq [X,2X]$ is  a $X^{1-(1-2\delta)(2\pi \varepsilon)^{1/2}}$ seperated set satisfying
\begin{align*}
|\psi(x)-x|\ge \varepsilon (\log{x})^2 x^{1/2}, \quad x\in \cX.
\end{align*}
We have
\begin{align*}
|\cX|\le \exp\left(\frac{C}{(\varepsilon\delta)^2}\right),
\end{align*}
for some absolute constant $C$.
\end{theorem}
Taking $\delta$ sufficiently small in Theorem~\ref{thm:main1} and using a dyadic decomposition, we obtain for any fixed $c'< \sqrt{2\pi }$
\begin{align}
\label{eq:xx1-zzzzzzzz}
\left|\left\{ x \le X \ : \ |\psi(x)-x|\ge \varepsilon X^{1/2}(\log{X})^{2}\right\}\right|\ll X^{1-c'\varepsilon^{1/2}}
\end{align}
which improves on Proposition~\ref{prop:largevalue}. 

\section{Bohr sets}
Given tuples of real numbers $\alpha=(\alpha_1,\dots,\alpha_k), \beta=(\beta_1,\dots,\beta_k)$ and a positive real numbers $\rho$, we define Bohr sets in the usual way
\begin{align}
\label{eq:BBBAT}
B(\alpha,\beta;\rho)=\{ x\in \R \ : \ \|\alpha_{\ell} x+\beta_\ell\|\le \rho, \quad 1\le \ell \le k\},
\end{align}
where $\|.\|$ denotes distance to the nearest integer.

We also consider truncated Bohr sets
\begin{align}
\label{eq:BaT}
B(\alpha,\beta,T;\rho)=\{ 0\le x \le T \ : \ \|\alpha_{\ell} x+\beta_\ell\|\le \rho, \quad 1\le \ell \le k\}.
\end{align}
We expect that 
\begin{align}
\label{eq:bbb}
\mu(B(\alpha,\beta,T;\rho))\approx(2\rho)^{k}T,
\end{align}
however this is not true in general. 

Our first result shows that~\eqref{eq:bbb} holds on average over $\alpha_1,\dots,\alpha_k$ satisfying suitable spacing conditions.
\begin{lemma}
\label{lem:Bohrcard}
Let $T\ge 1,C$ a sufficiently large constant, $\rho$ sufficiently small and $\cY\subseteq \R$
 a finite set satisfying 
\begin{align}
\label{eq:Yspacing}
|y-y'|\ge \frac{1}{T}, \quad \text{if} \quad y,y'\in \cY \quad \text{and} \quad y\neq y'.
\end{align}
For any $\eta>0$, we have 
\begin{align*}
&\sum_{(\alpha_1,\dots,\alpha_k)\in \cY^{k}}\max_{(\beta_1,\dots,\beta_k)\in \R^{k}}\mu(B(\alpha,\beta,T;\rho))\ll \\ & \quad \quad \quad \left( (2\rho)^{k}(1+\eta)^{k}+\frac{(C\log((\rho\eta)^{-1}))^{k}}{|\cY|}\right)T|\cY|^{k},
\end{align*}
 with implied constant independent of $k$.
\end{lemma}

\subsection{Smooth majorant for Bohr sets}
Our main tool to study Bohr sets is a smooth approximation to their indicator function. We first recall a construction of Vinogradov~\cite[Lemma~12; Chapter 1]{Vin}.
\begin{lemma}
\label{lem:Vin}
Let $r$ be a positive integer and $a,b,\Delta$ real numbers satisfying 
\begin{align*}
0<\Delta<\frac{1}{2}, \quad \Delta\le b-a\le 1-\Delta.
\end{align*}
There exists a periodic function $\Psi(x)$, with period $1$, satisfying 
\begin{align*}
&\Psi(x)=1 \quad \text{if} \quad a+\frac{\Delta}{2}\le x \le b -\frac{\Delta}{2}, \\ &\Psi(x)=0 \quad \text{if} \quad b+\frac{\Delta}{2}\le x \le 1+a-\frac{\Delta}{2}, \\
&0\le \Psi(x)\le 1 \quad \text{otherwise},
\end{align*}
with an expansion into Fourier series
\begin{align*}
\Psi(x)=\sum_{\substack{m\in \Z }}a_me(mx),
\end{align*}
where $a_m$ satisfies
\begin{align*}
a_0&=b-a,  \quad 
|a_m|\le 2(b-a), \quad
|a_m|\le \frac{2}{\pi |m|}, \quad  
|a_m|\le \frac{2}{\pi |m|}\left(\frac{r}{\pi |m| \Delta}\right)^r.
\end{align*}
\end{lemma}

\begin{lemma}
\label{lem:Bohrma}
Let $\alpha=(\alpha_1,\dots,\alpha_k),(\beta_1,\dots,\beta_k)\in \R^{k}$ and suppose $\rho,\eta$ are sufficiently small.

 For any integer $r$, there exists a function $\Psi_{\alpha,\beta}$ satisfying
\begin{align}
\label{eq:psi01}
0\le \Psi_{\alpha,\beta}(x)\le 1,
\end{align}
\begin{align}
\label{eq:psi1}
\Psi_{\alpha,\beta}(x)=1 \quad \text{if} \quad x\in B(\alpha,\beta;\rho),
\end{align}
with expansion into a trigonometric series 
\begin{align}
\label{eq:PsiF}
\Psi_{\alpha,\beta}(x)=\sum_{m_1,\dots,m_k\in \Z}a^{(1)}_{m_1}\dots a^{(k)}_{m_k}e((\alpha_1m_1+\dots+\alpha_km_k)x),
\end{align}
where each $a^{(j)}_m$ satisfies 
\begin{align}
\label{eq:acoeffs}
a^{(j)}_0&=2\rho(1+\eta),  \quad 
|a^{(j)}_m|\le 4\rho(1+\eta), \quad \\ &
|a^{(j)}_m|\le \frac{2}{\pi |m|}, \quad   \nonumber
|a^{(j)}_m|\le \frac{2}{\pi |m|}\left(\frac{r}{2\pi |m| \rho\eta}\right)^r.
\end{align}
\end{lemma}
\begin{proof}
 Let $\Psi$ be as in Lemma~\ref{lem:Vin} with parameters 
\begin{align*}
a=-\rho(1+\eta), \quad b=\rho(1+\eta), \quad \Delta=2\rho\eta,
\end{align*}
 and define
\begin{align}
\label{eq:Psidef}
\Psi_{\alpha,\beta}(x)=\prod_{\ell=1}^{k}\Psi(\alpha_{\ell}x+\beta_{\ell}).
\end{align}
For any $x\in \R$
\begin{align*}
0\le \Psi_{\alpha,\beta}(x)\le 1.
\end{align*}
If $x$ satisfies 
\begin{align*}
\|\alpha_{\ell}x+\beta_{\ell}\|\le \rho, \quad 1\le \ell \le k,
\end{align*}
then 
\begin{align*}
\Psi(\alpha_{\ell}x+\beta_{\ell})=1, \quad 1\le \ell \le k, 
\end{align*}
and hence 
\begin{align*}
\Psi_{\alpha,\beta}(x)=1.
\end{align*}
This establishes~\eqref{eq:psi01} and~\eqref{eq:psi1}. Expanding each factor in~\eqref{eq:Psidef} into a Fourier series gives~\eqref{eq:PsiF} and~\eqref{eq:acoeffs}.
\end{proof}
\section{Proof of Lemma~\ref{lem:Bohrcard}}
Fix $\alpha=(\alpha_1,\dots,\alpha_k)\in \cY^{k}$ and $(\beta_1,\dots,\beta_k)\in \R^{k}$ and consider  $B(\alpha,\beta,T;\rho).$

Let $\Phi_{\alpha,\beta}$ be as in Lemma~\ref{lem:Bohrma} with $r=2$ and suppose $f$ is a positive smooth function satisfying 
$$f(x)\gg 1 \quad \text{if} \quad |x|\le 1,$$
and 
\begin{align}
\label{eq:fhathat}
\text{supp}(\widehat f)\subseteq [-1,1].
\end{align}
We have 
\begin{align*}
\mu(B(\alpha,\beta,T;\rho))\ll \int_{-\infty}^{\infty}f\left(\frac{t}{T}\right)\Phi_{\alpha,\beta}(t)dt.
\end{align*}
By~\eqref{eq:PsiF} and~\eqref{eq:fhathat}
\begin{align*}
\mu(B(\alpha,\beta,T;\rho))&\ll \sum_{m_1,\dots,m_k\in \Z}a^{(1)}_{m_1}\dots a^{(k)}_{m_k}\int_{-\infty}^{\infty}f\left(\frac{t}{T}\right)e((\alpha_1m_1+\dots+\alpha_km_k)t)dt \\ 
&\ll T\sum_{m_1,\dots,m_k\in \Z}a^{(1)}_{m_1}\dots a^{(k)}_{m_k}\widehat f((\alpha_1m_1+\dots+\alpha_km_k)T) \\ 
&\ll (2\rho)^{k}(1+\eta)^{k}T+T\sum_{\substack{m_1,\dots,m_k\in \Z \\ (m_1,\dots,m_k)\neq 0 \\ |m_1\alpha_1+\dots+m_k\alpha_k|\le 2\pi /T }}a'_{m_1}\dots a'_{m_k},
\end{align*}
where 
\begin{align}
\label{eq:a'def}
a'_m=\min \left\{ 4\rho(1+\eta),\frac{2}{\pi |m|}, \frac{2}{\pi |m|}\left(\frac{r}{\pi |m| \rho \eta}\right)^2 \right\}.
\end{align}
Summing the above over $(\alpha_1,\dots,\alpha_k)\in \cY^{k}$, we see that 
\begin{align*}
&\sum_{(\alpha_1,\dots,\alpha_k)\in \cY^{k}}\max_{(\beta_1,\dots,\beta_k)\in \R^{k}}\mu(B(\alpha,\beta,T;\rho))\ll  (2\rho)^{k}(1+\eta)^{k}T|\cY|^{k} \\ & \quad \quad \quad \quad +T\sum_{\substack{m_1,\dots,m_k\in \Z \\ (m_1,\dots,m_k)\neq 0 }}a'_{m_1}\dots a'_{m_k}\sum_{\substack{(\alpha_1,\dots,\alpha_k)\in \cY^{k} \\ |m_1\alpha_1+\dots+m_k\alpha_k|\le 2\pi /T}}1.
\end{align*}
Fix $(m_1,\dots,m_k)\neq 0$ and consider summation over $\alpha_1,\dots,\alpha_k$. Let $\ell$ satisfy $m_{\ell}\neq 0$. For each choice of $$\alpha_1,\dots,\alpha_{\ell-1},\alpha_{\ell+1},\dots,\alpha_{k}\in \cY,$$ there exists at most $O(1)$ values of $\alpha_{\ell}\in \cY$ satisfying 
\begin{align*}
|m_1\alpha_1+\dots+m_k\alpha_k|\le 2\pi /T.
\end{align*}
This implies 
\begin{align*}
\sum_{\substack{m_1,\dots,m_k\in \Z \\ (m_1,\dots,m_k)\neq 0 }}a'_{m_1}\dots a'_{m_k}\sum_{\substack{(\alpha_1,\dots,\alpha_k)\in \cY^{k} \\ |m_1\alpha_1+\dots+m_k\alpha_k|\le 2\pi /T}}1\ll |\cY|^{k-1}\left(\sum_{\substack{m\in \Z }}a'_m\right)^{k},
\end{align*}
and hence from~\eqref{eq:a'def}
\begin{align*}
&\sum_{(\alpha_1,\dots,\alpha_k)\in \cY^{k}}\max_{(\beta_1,\dots,\beta_k)\in \R^{k}}\mu(B(\alpha,\beta,T;\rho))\ll \\ & \quad \quad \quad \quad \left( (2\rho)^{k}(1+\eta)^{k}+\frac{(C\log((\rho\eta)^{-1}))^{k}}{|\cY|}\right)T|\cY|^{k},
\end{align*}
which completes the proof.

\section{Combinatorial decomposition}

\subsection{Constructing large exponential sums}
For $T>0$ define 
$$N(T)=|\{ 0\le \gamma \le T \ : \ \zeta(1/2+i\gamma)=0\}|,$$
and recall the Riemann-von Mangoldt formula~\cite[Corollary~14.2]{MV}
\begin{align}
\label{eq:rvm}
N(T)=\frac{(1+o(1))}{2\pi} T\log{T}.
\end{align}

\begin{lemma}
\label{lem:psi-to-gamma}
Let $X\le x \le 2X$ satisfy 
\begin{align}
\label{eq:psix1}
|\psi(x)-x|\ge \varepsilon x^{1/2}(\log{x})^2.
\end{align}
Let $\alpha, \beta$ satisfy 
\begin{align*}
0<\alpha,\beta<1.
\end{align*}
Define
$$\cT_{x}=\left\{ X^{\alpha (2\pi \varepsilon)^{1/2}}\le t \le (\log{X})X^{1/2} \ : \ \left|\sum_{ \gamma \le t}x^{i\gamma}\right|\ge 8\pi \varepsilon \beta N(t) \right\},$$
and 
\begin{align}
\label{eq:delta-def}
\delta(\alpha,\beta)=1-\alpha^2-\beta+o(1).
\end{align}
We have
\begin{align*}
\int_{\cT_{x}}\frac{1}{t}dt\ge 2\pi \varepsilon\delta(\alpha,\beta)\log{X}.
\end{align*}
\end{lemma}
\begin{proof}
By~\cite[Theorem~12.5]{MV}
\begin{align*}
\psi(x)-x=x^{1/2}\sum_{|\gamma|\le T}\frac{x^{i\gamma}}{1/2+i\gamma}+O\left(\frac{x^{1/2}(\log{xT})^2}{T}\right).
\end{align*}
Taking 
$$T=(\log{X})X^{1/2},$$
 using the assumption~\eqref{eq:psix1} and partitioning summation over $\gamma$ depending on if $\gamma>0$ or not, we see that
\begin{align*}
\left|\sum_{0<\gamma \le (\log{X})X^{1/2}}\frac{x^{i\gamma}}{1/2+i\gamma}\right|\ge \frac{(1+o(1))\varepsilon}{2}(\log{X})^2.
\end{align*}
By partial summation 
\begin{align}
\label{eq:lb-1}
\int_{1/2}^{(\log{X})X^{1/2}}\frac{1}{t^2}\left|\sum_{0<\gamma \le t}x^{i\gamma}\right|dt\ge \frac{(1+o(1))\varepsilon}{2}(\log{X})^2.
\end{align}
By~\eqref{eq:rvm} we see that
\begin{align}
\label{eq:lb-2}
\int_{1/2}^{X^{\alpha (2\pi \varepsilon)^{1/2}}}\frac{1}{t^2}\left|\sum_{0<\gamma \le t}x^{i\gamma}\right|dt & \le \int_{1/2}^{X^{\alpha (2\pi \varepsilon)^{1/2}}}\frac{N(t)}{t^2}dt \nonumber \\ 
& \le \frac{(1+o(1))}{2\pi }\int_{1/2}^{X^{\alpha (2\pi \varepsilon)^{1/2}}}\frac{\log{t}}{t}dt \nonumber \\ 
&\le \frac{(1+o(1))}{2 }\alpha^2\varepsilon
\end{align}
and 
\begin{align}
\label{eq:lb-3}
\nonumber \int_{\substack{1/2\le t \le (\log{X})X^{1/2} \\ t\not\in \cT_x}}\frac{1}{t^2}\left|\sum_{0<\gamma \le t}x^{i\gamma}\right|dt&\le 8\pi \beta \varepsilon \int_{1/2}^{ (\log{X})X^{1/2} }\frac{N(t)}{t^2}dt  \\ 
& \le \frac{\beta \varepsilon}{2}(\log{X})^2.
\end{align}
Combining~\eqref{eq:lb-1},~\eqref{eq:lb-2} and~\eqref{eq:lb-3} gives
\begin{align*}
\int_{t\in \cT_{x}}\frac{N(t)}{t^2}dt\ge \int_{t\in \cT_{x}}\frac{1}{t^2}\left|\sum_{0<\gamma \le t}x^{i\gamma}\right|dt\ge \frac{\varepsilon(1-\alpha^2-\beta+o(1))}{2}(\log{X})^2.
\end{align*}
By~\eqref{eq:rvm} we have
\begin{align*}
\int_{t\in \cT_{x}}\frac{N(t)}{t^2}dt\le \frac{(1+o(1))\log{X}}{4\pi}\int_{t\in \cT_{x}}\frac{1}{t}dt,
\end{align*}
from which the result follows.
\end{proof}
We next find a value of $T$ such that the sums 
$$\left|\sum_{0<\gamma \le T}x^{i\gamma}\right|,$$
are large for many values of $x$.
\begin{lemma}
\label{lem:large-scale}
Let $\cX\subseteq [X,2X]$ satisfy 
\begin{align*}
|\psi(x)-x|\ge \varepsilon x^{1/2}(\log{x})^{2}, \quad x\in \cX,
\end{align*}
Let  $\alpha,\beta,\delta(\alpha,\beta)$ be as in Lemma~\ref{lem:psi-to-gamma} and suppose that 
\begin{align}
\label{eq:deltaconds}
\delta(\alpha,\beta)>0.
\end{align}
There exists 
\begin{align}
\label{eq:Tconds}
X^{\alpha(2\pi\varepsilon)^{1/2}}\le T \le (\log{X})X^{1/2},
\end{align}
and 
\begin{align}
\cX_0\subseteq \cX, \quad |\cX_0|\ge \varepsilon \delta(\alpha,\beta) |\cX|,
\end{align}
such that 
\begin{align*}
\left|\sum_{0<\gamma \le T}x^{i\gamma}\right|\ge 8\pi \beta \varepsilon N(T), \quad x\in \cX_0.
\end{align*}
\end{lemma}
\begin{proof}
By Lemma~\ref{lem:psi-to-gamma}
\begin{align}
\label{eq:aaaaaaaaaaaaaaa569594356}
& \int_{X^{\alpha(2\pi \varepsilon)^{1/2}}\le t \le (\log{X})X^{1/2}}\frac{|\{x\in \cX \ : \ t\in \cT_{x}\}|}{t}dt \\ & \quad \quad \quad \ge \sum_{x\in \cX}\int_{t\in \cT_{x}}\frac{1}{t}dt  \ge 2\pi \varepsilon \delta(\alpha,\beta)|\cX|\log{X}, \nonumber
\end{align}
which after taking a maximum over $X^{\alpha(2\pi \varepsilon)^{1/2}}\le t \le (\log{X})X^{1/2}$ implies 
\begin{align*}
&\max_{X^{\alpha (2\pi \varepsilon)^{1/2}}\le t \le (\log{X})X^{1/2}}|\{x\in \cX \ : \ t\in \cT_{x}\}|\int_{X^{\alpha(2\pi \varepsilon)^{1/2}}\le t \le (\log{X})X^{1/2}}\frac{1}{t}dt \\ & \quad \quad \quad  \ge \varepsilon \pi \delta(\alpha,\beta)|\cX|\log{X},
\end{align*}
and we obtain the desired result.
\end{proof}
Our next result gives a partition the zeros of $\zeta$ into two sets, one regular and one with small cardinality.
\begin{lemma}
\label{lem:partition}
Let
$$\cN=\{ \gamma \ : \ \zeta(1/2+i\gamma)=0  \},$$
with $\gamma$ counted according to multiplicity. For each $K\ge 1$ there exists a disjoint partition 
$$\cN=\cN_1\bigcup \cN_2,$$
such that 
\begin{align}
\label{eq:N2}
|\cN_2\cap [0,T]|\le \frac{1}{200 K}N(T),
\end{align}
and for each interval $\cI\subseteq [0,T]$ satisfying 
$$|\cI|\le \frac{1}{\log{T}},$$
we have 
\begin{align}
\label{eq:N1}
|\cN_1\cap \cI|\ll K.
\end{align}
\end{lemma}
\begin{proof}
Let $\cJ$ be a disjoint partition of $[0,T]$ into intervals of length $1/\log{T}$, 
so that 
$$\cN=\bigcup_{\cI\in \cJ}\cN\cap \cI.$$
Define 
$$\cJ_1=\left\{ \cI \ : \ |\cN\cap \cI|\le CK \right\}, \quad \cJ_2=\left\{ \cI \ : \ |\cN\cap \cI|> CK \right\}$$
and 
$$\cN_1=\bigcup_{\cI\in \cJ_1}\cN\cap \cI, \quad \cN_2=\bigcup_{\cI\in \cJ_2}\cN\cap \cI.$$
The set$\cN_1$ satisfies~\eqref{eq:N1} by construction. 

It remains to establish~\eqref{eq:N2}.
We start by observing that 
\begin{align*}
CK|\cN_2|\le \sum_{\cI\in \cJ_2}|\cN\cap \cI|^2\le \left|\left\{ 0\le \gamma,\gamma'\le T \ :  |\gamma-\gamma'|\le \frac{2}{\log{T}} \right\}\right|.
\end{align*}
It follows from~\cite[Corollary~1]{Mont} that
\begin{align*}
\left|\left\{ 0\le \gamma,\gamma'\le T \ :  |\gamma-\gamma'|\le \frac{2}{\log{T}} \right\}\right|\ll T\log{T}.
\end{align*}
Combining the above with~\eqref{eq:rvm} and choosing $C$ suitably gives
\begin{align*}
|\cN_2|\le \frac{1}{200 K}N(T),
\end{align*}
which establishes the desired result.
\end{proof}

\subsection{Concentration into Bohr sets}
We next concentrate $\gamma$'s into Bohr sets via large exponential sums. Results of this sort are well known.
\begin{lemma}
\label{lem:concentration}
Let $T\gg 1$ and 
$x_1,\dots,x_N \in [0,T]$. Suppose $\alpha$ satisfies 
\begin{align}
\label{eq:alphaconds}
\alpha T\ge \frac{2}{\delta}
\end{align}
and
\begin{align}
\label{eq:largesum}
\left|\sum_{j=1}^{N}e(\alpha x_j)\right|\ge \delta N.
\end{align}
Let $\varepsilon>0$ satisfy 
\begin{align}
\label{eq:varepsilonconds}
\varepsilon<\frac{\delta}{C},
\end{align}
for an absolute constant $C$. There exists some $0\le \beta < 1$ such that 
\begin{align}
\label{eq:Bohrdensity}
|\{ 1\le j \le N \ : \ \|\alpha x_j+\beta\|\le \varepsilon\}|\ge 2\varepsilon(1+\frac{\delta}{16})N.
\end{align}
\end{lemma}
\begin{proof}
Let $\gamma$ be sufficiently small and define the function
\begin{align}
\label{eq:fdef}
f(x)=\frac{1}{\gamma}\left(\sum_{n=1}^{N}1_{\gamma,x_n}(x)-\frac{\gamma N}{T}1_{[0,T]}(x)\right),
\end{align}
where $1_{\gamma,x_n}$ denotes the indicator function of the interval $[x_n,x_n+\gamma]$ and $1_{[0,T]}$ denotes the indicator function of the interval $[0,T]$. We have  
\begin{align}
\label{eq:hatf}
\widehat f(\alpha)&=\int_{0}^{T}f(x)e(\alpha x)dx=\sum_{n=1}^{N}\frac{1}{\gamma}\int_{x_n}^{x_n+\gamma}e(\alpha x)dx-\frac{N}{T}\int_{0}^{T}e(\alpha x)dx.
\end{align}
As $\gamma$ tends to zero
\begin{align*}
\frac{1}{\gamma}\int_{x_n}^{x_n+\gamma}e(\alpha x)dx=(1+o(1))e(\alpha x_n),
\end{align*}
which implies 
\begin{align}
\label{eq:hatf1}
\left|\sum_{n=1}^{N}\frac{1}{\gamma}\int_{x_n}^{x_n+\gamma}e(\alpha x)dx\right|\ge \frac{3\delta}{4}N,
\end{align}
after taking $\gamma$ sufficiently small. Using
\begin{align*}
\left|\int_{0}^{T}e(\alpha x)dx\right|\le \frac{1}{\pi \alpha},
\end{align*}
we see that~\eqref{eq:alphaconds} implies
\begin{align*}
\frac{N}{T}\left|\int_{0}^{T}e(\alpha x)dx\right|\le \frac{\delta}{4}N.
\end{align*}
Combining the above with~\eqref{eq:hatf} and~\eqref{eq:hatf1} shows that
\begin{align}
\label{eq:largef1}
\left|\widehat f(\alpha)\right|\ge \frac{\delta}{2}N.
\end{align}
Consider 
\begin{align*}
S=\int_{0}^{1}\left(\int_{\substack{0\le x \le T \\ \|\alpha x+\beta\|\le \varepsilon}}f(x)e(\alpha x)dx\right)d\beta.
\end{align*}
Interchanging the order of integration gives
\begin{align*}
S=\int_{\substack{0\le x \le T}}f(x)e(\alpha x)\int_{\substack{0\le \beta \le 1 \\ \|\alpha x+\beta\|\le \varepsilon}}1 d\beta dx=2\varepsilon \widehat f(\alpha).
\end{align*}
Using~\eqref{eq:largef1}, this implies that 
\begin{align}
\label{eq:SLB}
|S|\ge (2\varepsilon)\frac{\delta}{2}N.
\end{align}
Since the phase in integration over $x$ is roughly constant on each fixed $\beta$, we have 
\begin{align*}
|S|&\le \int_{0}^{1}\left|\int_{\substack{0\le x \le T \\ \|\alpha x+\beta\|\le \varepsilon}}f(x)e(\alpha x+\beta)dx \right|d\beta \\ 
&=\int_{0}^{1}\left|\int_{\substack{0\le x \le T \\ \|\alpha x+\beta\|\le \varepsilon}}f(x)dx\right|d\beta+O\left(\varepsilon \int_{0}^{1}\int_{\substack{0\le x \le T \\ \|\alpha x+\beta\|\le \varepsilon}}|f(x)|dxd\beta \right) \\ 
&=\int_{0}^{1}\left|\int_{\substack{0\le x \le T \\ \|\alpha x+\beta\|\le \varepsilon}}f(x)dx\right|d\beta+O\left(\varepsilon^2N\right).
\end{align*}
Using~\eqref{eq:varepsilonconds},~\eqref{eq:SLB} and the fact that 
\begin{align*}
\int_{0}^{1}\int_{\substack{0\le x \le T \\ \|\alpha x+\beta\|\le \varepsilon}}f(x)dxd\beta=0,
\end{align*}
the above implies 
\begin{align*}
\int_{0}^{1}\left|\int_{\substack{0\le x \le T \\ \|\alpha x+\beta\|\le \varepsilon}}f(x)dx\right|+\left(\int_{\substack{0\le x \le T \\ \|\alpha x+\beta\|\le \varepsilon}}f(x)dx\right)d\beta\ge (2\varepsilon)\frac{\delta}{4}N.
\end{align*}
Hence there exists some $0\le \beta <1$ such that 
\begin{align*}
\int_{\substack{0\le x \le T \\ \|\alpha x+\beta\|\le \varepsilon}}f(x)dx\ge (2\varepsilon)\frac{\delta}{8}N.
\end{align*}
Recalling~\eqref{eq:fdef} and letting $\gamma$ tend to zero
\begin{align*}
&|\{ 1\le j \le N \ : \ \|\alpha x_j+\beta\|\le \varepsilon\}| \\ & \quad \quad \quad \ge \frac{N}{T}\mu(\{ 0\le t \le T \ : \ \|\alpha t+\beta\|\le \varepsilon\})+(2\varepsilon)\frac{\delta}{8}N,
\end{align*}
from which the result follows, since 
\begin{align*}
\mu(\{ 0\le t \le T \ : \ \|\alpha t+\beta\|\le \varepsilon\})=2\varepsilon (T+O(1)).
\end{align*}
\end{proof}
Summarising our progress thus far, we have:
\begin{lemma}
\label{prop:main-reduction}
Let $\cY\subseteq [X,2X]$ be a finite set satisfying
$$ |\psi(x)-x|\ge \varepsilon x^{1/2}(\log{x})^2 \quad x\in \cY.$$
Let $\delta(\alpha,\beta)$ be given by~\eqref{eq:delta-def} and suppose $0<\alpha,\beta<1$ satisfy 
\begin{align*}
\delta(\alpha,\beta)>0.
\end{align*}
 There exists   $T$ satisfying 
\begin{align}
\label{eq:mainT}
X^{\alpha(2\pi\varepsilon)^{1/2}}\le T \le (\log{X})X^{1/2},
\end{align}
 a subset
$$\cN_1\subseteq \{ 0\le \gamma \le T \ : \zeta(1/2+i\gamma)=0\},$$
satisfying:
\begin{enumerate} 
\item 
\begin{align*}
|\cN_1|\ge  \left(1-\frac{\varepsilon \beta}{200}\right)N(T),
\end{align*}
and 
$$|\cN_1\cap \cI|\ll \frac{1}{\varepsilon \beta},$$
for each interval $|\cI|\le 1/\log{T}.$
\item Some $\cX\subseteq \cY$ satisfying
\begin{align}
\label{eq:Xcard}
|\cX|\ge \varepsilon \delta(\alpha,\beta)|\cY|,
\end{align}
and for each $x\in \cX$, some $\beta(x)$ such that 
\begin{align*}
\left|\left\{ \gamma \in \cN_1 : \ \left\|\frac{\log{x}}{2\pi} \gamma+\beta(x) \right\|\le \rho \right\} \right|\ge 2\rho \left(1+\frac{\varepsilon \beta}{1000}\right)|\cN_1|.
\end{align*}
\end{enumerate}
\end{lemma}
\begin{proof}

By Lemma~\ref{lem:large-scale}, there exists 
\begin{align*}
X^{\alpha(2\pi\varepsilon)^{1/2}}\le T \le (\log{X})X^{1/2},
\end{align*}
and $\cX\subseteq \cY$ satisfying 
\begin{align*}
|\cX|\ge \frac{\varepsilon \delta(\alpha,\beta)}{100}|\cY|,
\end{align*}
such that
\begin{align*}
\left|\sum_{0<\gamma\le T}x^{i\gamma}\right|\ge \frac{\varepsilon \beta}{10}N(T), \quad x\in \cX.
\end{align*}
By Lemma~\ref{lem:partition}, there exists 
$$\cN_1\subseteq \{ 0\le \gamma \le T \ :  \zeta(1/2+i\gamma)=0 \},$$ 
satisfying 
\begin{align}
\label{eq:N1card}
|\cN_1|\ge \left(1-\frac{\varepsilon \beta}{200}\right)N(T),
\end{align}
\begin{align*}
\left|\sum_{\gamma \in \cN_1}x^{i\gamma}\right|\ge \frac{\varepsilon \beta}{100}N(T), \quad x\in \cX,
\end{align*}
and for any interval $|\cI|\le 1/\log{T}$
$$|\cN_1\cap \cI|\ll \frac{1}{\varepsilon \beta}.$$
By Lemma~\ref{lem:concentration} for any 
$$\rho\le \frac{\varepsilon}{C},$$
and each $x\in \cX$, there exists some $\beta(x)$ such that 
\begin{align*}
\left|\left\{ \gamma \in \cN_1 : \ \left\|\frac{\log{x}}{2\pi} \gamma+\beta(x) \right\|\le \rho \right\} \right|\ge 2\rho \left(1+\frac{\varepsilon \beta}{1000}\right)|\cN_1|,
\end{align*}
from which the result follows.
\end{proof}
Lemma~\ref{prop:main-reduction} implies we can concentrate $\gamma$'s into many rank one Bohr sets. We next use H\"{o}lder's inequality to find large rank Bohr sets containing many zeros. 
\begin{lemma}
\label{lem:largerank}
Let notation and conditions be as in Lemma~\ref{prop:main-reduction}.  For any integer $k$ 
we have 
\begin{align*}
(2\rho)^{k}\left(1+\frac{\varepsilon \beta}{1000}\right)^{k}|\cX|^{k}|\cN_1|&\le \sum_{x_1,\dots,x_k\in \cX}\max_{(\beta_1,\dots,\beta_k)\in \R^{k}}\sum_{\substack{\gamma \in \cN_1 \\ \|\log{x_j}\gamma/2\pi +\beta_j\|\le \rho \\ 1\le j \le k}}1.
\end{align*}
\end{lemma}
\begin{proof}
By Lemma~\ref{prop:main-reduction}
\begin{align*}
(2\rho)\left(1+\frac{\varepsilon \beta}{1000}\right)|\cN_1||\cX|\le \sum_{\substack{x\in \cX}}\sum_{\substack{\gamma \in \cN_1 \\ \|\log{x}\gamma/2\pi +\beta(x)\|\le \rho}}1.
\end{align*}
Interchanging summation, applying H\"{o}lder's inequality then interchanging summation again
\begin{align*}
(2\rho)^{k}\left(1+\frac{\varepsilon \beta}{1000}\right)^{k}|\cX|^{k}|\cN_1|&\le \sum_{\gamma \in \cN_1}\left(\sum_{\substack{x\in \cX \\ \|\log{x}\gamma/2\pi +\beta(x)\|\le \rho}}1 \right)^{k} \\ 
&=\sum_{x_1,\dots,x_k\in \cX}\sum_{\substack{\gamma \in \cN_1 \\ \|\log{x_j}\gamma/2\pi +\beta(x_j)\|\le \rho \\ 1\le j \le k}}1,
\end{align*}
and the result follows after taking a maximum over $\beta(x_1),\dots,\beta(x_k)$.
\end{proof}
\subsection{Counting sequences in Bohr sets}
We next show how to count sequences of real numbers in Bohr sets via concentration into short intervals. Recall notation~\eqref{eq:BBBAT} and~\eqref{eq:BaT}
\begin{lemma}
\label{lem:bohr-sequence}
Let $T\ge 1$, $\alpha=(\alpha_1,\dots,\alpha_k)\in \R^{k}$ satisfy 
\begin{align}
\label{eq:alphajbound}
|\alpha_j|\le A\log{T}, \quad 1\le j \le k,
\end{align}
for some $A>0$. Let $\beta=(\beta_1,\dots,\beta_k)\in \R^{k}$ and $\rho>0$. For any finite set $\cA\subseteq [0,T]$  and $\eta>0$  we have 
\begin{align*}
|\cA\cap B(\alpha,\beta;\rho)|&\ll \frac{A\log{T}}{\eta \rho}\max_{\substack{\cI \ \ \text{interval} \\ |\cI|=2\eta \rho/(A\log{T})}}|\cA\cap \cI| \\ 
& \times  \max_{(\beta_1,\dots,\beta_k)\in \R^{k}}\mu(B(\alpha,\beta,T;\rho(1+\eta))).
\end{align*}
\end{lemma}
\begin{proof}
Suppose $\gamma\in B(\alpha,\beta,\rho)$. In particular, for each $1\le j \le k$ we have 
\begin{align*}
\|\alpha_j\gamma+\beta_j\|\le \rho.
\end{align*}
By~\eqref{eq:alphajbound}, for any 
$$|x|\le \frac{\eta \rho}{A\log{T}},$$
we have 
\begin{align*}
\|\alpha_j(\gamma+x)+\beta_j\|\le \rho(1+\eta).
\end{align*}
This implies 
$$\gamma+x\in B(\alpha,\beta,\rho(1+\eta)).$$
Let $1_B$ denote the indicator function of $B(\alpha,\beta;\rho(1+\eta))$. From the above
\begin{align*}
&\frac{2\eta \rho}{A\log{T}}|\cA\cap B(\alpha,\beta;\rho)|= \sum_{\gamma \in \cA \cap B(\alpha,\beta;\rho)}\int_{-\eta \rho/(A\log{T})}^{\eta \rho/(A\log{T})}1_B(\gamma+x)dx \\ 
&\le \int_{-\eta \rho/(A\log{T})}^{T+\eta \rho/(A\log{T})}|\{ \gamma \in \cA \ : \ |\gamma-t|\le \eta \rho/(A\log{T})\}|1_B(t)dt \\
& \le \max_{\substack{\cI \ \ \text{interval} \\ |\cI|=2\eta \rho/(A\log{T})}}|\cA\cap \cI| \\ & \times \mu(B(\alpha,\beta;\rho(1+\eta))\cap [-\eta \rho/(A\log{T}),T+\eta\rho/(A\log{T})])
\end{align*}
and hence 
\begin{align*}
|\cA\cap B(\alpha,\beta;\rho)|&\ll \frac{A\log{T}}{\eta \rho}\max_{\substack{\cI \ \ \text{interval} \\ |\cI|=2\eta \rho/(A\log{T})}}|\cA\cap \cI| \\ 
& \times  \max_{(\beta_1,\dots,\beta_k)\in \R^{k}}\mu(B(\alpha,\beta,T;\rho(1+\eta))),
\end{align*}
which completes the proof.
\end{proof}
\section{Proof of Theorem~\ref{thm:main1}}
In order to obtain a contradiction, assume there exists a $X^{1-(1-2\delta)(2\pi \varepsilon)^{1/2}}$-spaced set $\cX$ satisfying 
\begin{align}
\label{eq:ttc1}
|\cX|\ge \exp\left(\frac{C}{(\varepsilon\delta)^2}\right)
\end{align}
and 
\begin{align*}
|\psi(x)-x|\ge \varepsilon (\log{x})^2x^{1/2}, \quad x\in \cX,
\end{align*}
for some absolute constant $C$. Apply Lemma~\ref{prop:main-reduction} with
\begin{align}
\label{eq:mainrhodef}
\alpha=1-\delta, \quad \beta=\delta, \quad \rho=\frac{\varepsilon}{C},
\end{align}
for a suitably large constant $C$.
We see that there exists $T$ satisfying
\begin{align}
\label{eq:mainT}
X^{(1-\delta)(2\pi \varepsilon)^{1/2}}\le T \le (\log{X})X^{1/2},
\end{align}
a subset $\cN_1\subseteq \{ 0\le \gamma \le T \ : \ \zeta(1/2+i\gamma)=0\}$ satisfying 
\begin{align}
\label{eq:mainNN1conds}
|\cN_1|\ge \left(1-\frac{\varepsilon \delta}{200}\right)N(T), \quad |\cN_1\cap \cI|\ll \frac{1}{\varepsilon \delta},
\end{align}
for each interval $|\cI|\le 1/(\log{T})$, by~\eqref{eq:ttc1} some $\cX_1\subseteq \cX$ satisfying 

\begin{align}
\label{eq:mainX1lb}
|\cX_1| \ge \varepsilon \delta \exp\left(\frac{C}{(\varepsilon \delta)^2}\right),
\end{align}
and for each $x\in \cX_1$ some $\beta(x)\in \R$ satisfying 
\begin{align*}
\left|\left\{ \gamma \in \cN_1 : \ \left\|\frac{\log{x}}{2\pi} \gamma+\beta(x) \right\|\le \rho \right\} \right|\ge 2\rho \left(1+\frac{\varepsilon \delta}{1000}\right)|\cN_1|.
\end{align*}
In particular, for any $k\ge 1$, by Lemma~\ref{lem:largerank}
\begin{align}
\label{eq:XXXXX}
(2\rho)^{k}\left(1+\frac{\varepsilon \delta}{1000}\right)^{k}|\cX|^{k}|\cN_1|&\le \sum_{x_1,\dots,x_k\in \cX_1}\max_{(\beta_1,\dots,\beta_k)\in \R^{k}}\sum_{\substack{\gamma \in \cN_1 \\ \|\log{x_j}\gamma/2\pi +\beta_j\|\le \rho \\ 1\le j \le k}}1.
\end{align}
For fixed $x_1,\dots,x_k\in \cX_1$, we apply Lemma~\ref{lem:bohr-sequence} with parameters 
\begin{align}
\label{eq:maineta}
\eta=\frac{\varepsilon \delta}{10000}, \quad A=\frac{2\log{X}}{\log{T}}.
\end{align}
Note that by~\eqref{eq:mainT}
\begin{align}
\label{eq:mainAub}
A\ll \frac{1}{\varepsilon^{1/2}}.
\end{align}
By~\eqref{eq:mainrhodef},~\eqref{eq:mainNN1conds} and Lemma~\ref{lem:bohr-sequence} 
\begin{align*}
\max_{(\beta_1,\dots,\beta_k)\in \R^{k}}\sum_{\substack{\gamma \in \cN_1 \\ \|\log{x_j}\gamma/2\pi +\beta_j\|\le \rho \\ 1\le j \le k}}1\ll \frac{\log{T}}{\varepsilon^{7/2}\delta^2}\max_{(\beta_1,\dots,\beta_k)\in \R^{k}}\mu(B(\tilde{x},\beta,T;\rho(1+\eta))),
\end{align*}
where 
$$\tilde x=\left(\frac{\log{x_1}}{2\pi} ,\dots,\frac{\log{x_k}}{2\pi}\right).$$
By~\eqref{eq:XXXXX}
\begin{align}
\label{eq:main1b4} \nonumber
&(2\rho)^{k}\left(1+\frac{\varepsilon \varepsilon_1}{1000}\right)^{k}|\cX|^{k}|\cN_1| \\ & \ll \frac{\log{T}}{\varepsilon^{7/2}\delta^2}\sum_{x_1,\dots,x_k\in \cX_1}\max_{(\beta_1,\dots,\beta_k)\in \R^{k}}\mu(B(\tilde{x},\beta,T;\rho(1+\eta))).
\end{align}
We next apply Lemma~\ref{lem:Bohrcard} to summation over $\tilde x$ in~\eqref{eq:main1b4}. We first verify the condition~\eqref{eq:Yspacing} is satisfied with parameter $\cY=\cX_1$. By assumption, $\cX_1$ is $X^{1-(1-2\delta)(2\pi \varepsilon)^{1/2}}$-spaced, so for any distinct $x_1,x_2\in \cX_1$, we have 
\begin{align*}
\left|\frac{\log{x_1}}{2\pi}-\frac{\log{x_2}}{2\pi}\right|&\gg \log{\left(1+X^{-(1-2\delta)(2\pi \varepsilon)^{1/2}}\right)}  \\ 
&\gg \frac{1}{X^{(1-2\delta)(2\pi \varepsilon)^{1/2}}},
\end{align*}
where we have used the fact that $\cX\subseteq [X,2X]$. By~\eqref{eq:mainT}, the condition~\eqref{eq:Yspacing} is satisfied. Applying Lemma~\ref{lem:Bohrcard} with 
\begin{align}
\eta=\frac{\varepsilon \delta}{10000},
\end{align}
gives
\begin{align*}
&\sum_{x_1,\dots,x_k\in \cX_1}\max_{(\beta_1,\dots,\beta_k)\in \R^{k}}\mu(B(\tilde{x},\beta,T;\rho(1+\eta)))\ll  \\ & \left( (2\rho)^{k}\left(1+\frac{\varepsilon \delta}{10000}\right)^{2k}+\frac{(C\log((\varepsilon \delta)^{-1})^{k}}{|\cX_1|}\right)T|\cX_1|^{k}.
\end{align*}
Combining the above with~\eqref{eq:rvm},~\eqref{eq:mainNN1conds} and~\eqref{eq:main1b4}, we see that
\begin{align*}
\left(1+\frac{\varepsilon \delta}{1000}\right)^{k}\ll \frac{1}{\varepsilon^{7/2}\varepsilon^2_1}\left(\left(1+\frac{\varepsilon \delta}{10000}\right)^{2k}+\frac{(C'\log((\varepsilon \delta)^{-1})^{k}}{|\cX_1|\varepsilon^{k}}\right),
\end{align*}
for some absolute constant $C'$. Recalling~\eqref{eq:mainX1lb} and taking 
\begin{align*}
k=\frac{A\log{(\varepsilon \delta)^{-1}}}{(\varepsilon \delta)}, 
\end{align*}
for a sufficiently large constant $A$, the above implies 
\begin{align*}
\left(1+\frac{\varepsilon \delta}{5000}\right)^{k}\ll \frac{1}{\varepsilon^{7/2}\delta^2},
\end{align*}
from which we obtain a contradiction. 

\section{Acknowledgement}
The author would like to thank the Max Planck Institute for Mathematics and the Australian Research Council (DE220100859) for their support.

The author would like to thank Pieter Moree and Igor Shparlinski for a number of useful comments.

\end{document}